\documentclass[]{article}
\usepackage{amsmath,amsthm,amssymb}
\title{Ihara zeta function and twisted Alexander invariants}
\author{Zipei Zhuang}
\usepackage{indentfirst}
\usepackage{setspace}
\usepackage{mathrsfs}
\usepackage[backref]{hyperref}

\usepackage{graphicx}
\usepackage{import}
\usepackage{xifthen}
\usepackage{pdfpages}
\usepackage{transparent}

\begin{document}
	
	\maketitle

	\begin{abstract}
		In \cite{Lin98randomwalk}, Lin and Wang defined a model of random walks on knot diagrams and interprete the Alexnader polynomials and the colored Jones polynomials as Ihara zeta functions, i.e. zeta functions defined by counting cycles on the knot diagram. Using this explanation, they gave a more conceptual proof for the Melvin-Morton conjecture. In this paper, we give an analogous zeta function expression for the twisted Alexander invariants. 
	\end{abstract}
	\newtheorem{definition}{Definition}%
	\newtheorem{theorem}{Theorem}%
	\newtheorem{lemma}{Lemma}%
	\newtheorem{corollary}{Corollary}%
	\section{Introduction}

	Let $J_{K,n}(q) \in \mathbb{Z}[q^{\pm 1}]$ be the n-th colored Jones polynomial for the knot $K$ . Define
	\begin{equation}
 \begin{aligned}
	f_{K,n}: & \mathbb{C} \longrightarrow \mathbb{C} \\
    &z \longmapsto J_{K,n}(e^{z/n})
\end{aligned}
	\end{equation}

The Melvin-Morton conjecture (see \cite{MR2860990} \cite{Melvinmorton}) claims that
\begin{equation} \label{mm}
	\lim_{n \rightarrow \infty} f_{K,n}(z) = \frac{1}{\triangle_K (e^z)}
\end{equation}
 
\par 
 In \cite{Lin98randomwalk}, Lin and Wang introduced a random walk model on a knot diagram, and gave formulations for the Alexander polynomials and colored Jones polynomials in this model. More specifically, the Alexander polynomial equals the inverse of the Ihara zeta function of the knot diagram. On the other hand, the Jones polynomial is calculated by counting simple families of cycles on the diagram, and the colored Jones polynomials are obtained by  counting simple families of cycles on d-cabling of the link. In the limit of  (\ref{mm}), it is shown(see \cite{Lin98randomwalk} Section 3 ) that only weights of simple families of cycles survive, which leads to a proof of the Melvin-Morton conjecture.
 
 \par 
 The volume conjecture (see \cite{MR2860990} \cite{murakami2010introduction}) says that
 \begin{equation}
\lim_{n \rightarrow \infty} \frac{log|f_{K,n}(2 \pi i)|}{n} = \frac{1}{2 \pi}vol(K)
 \end{equation}
where $vol(K)$ denotes the simplicial volume of the 3-manifold $S^3-K$.

The $L^2$ -Alexander invariant,  of an oriented knot $K$,
\begin{equation}
\Gamma^{(2)}(K) : \mathbb{R}_{>0} \longrightarrow \mathbb{R}_{>0}
\end{equation} 
 introduced in \cite{Lizhang}, can be viewed as a "twisted" invariant of $K$ . It follows from the definition that (see\cite{https://doi.org/10.1007/978-3-662-04687-6}) that $\Gamma^{(2)}(K)(1)$ = the $L^2$-torsion of the knot exterior , which by work of $L\ddot{u}ck-Schick$ is nothing but the volume of the knot complement:

 \begin{equation}
\Gamma^{(2)}(K)(1) =\exp (\frac{1}{6 \pi} vol(K))
 \end{equation}
 Hence the volume conjecture can be written as 
 \begin{equation} \label{vc}
\lim_{n \rightarrow \infty} \frac{log|f_{K,n(2 \pi i)}|}{n} = \frac{1}{2 \pi}vol(K) =3 log\Gamma^{(2)}(K)(1)
 \end{equation}
Comparing (\ref{mm})(\ref{vc}), we wonder if there is also a zeta function formula for the $L^2$ -Alexander torsion. On the other hand, the zeta-function-expressions for knot polynomials seem to be interesting enough to deserve a study on its own right. 

In Section 2, we use the matrix tree theorem for oriented diagrams to give a new perspective for the Alexander polynomial.
In Section 3, we reviewed the concept of the arc diagram of a tangle, the Ihara zeta function of a diagram and the determinant formula;  In Section 4, we explain the Ihara zeta function expression of the Alexander polynomial and deduce some basic properties from this viewpoint; In Section 5, possible generalizations to twisted Alexander invariants are discussed.

 \section{Matrix tree theorem for oriented graphs}
 We reviewed the Fox-calculus definition for the Alexander polynomial. Suppose the knot group $G(K)=\pi_1 (S^3 -K)$ has the Wirtinger presentation
 \begin{equation}
 	G(K)=<x_1 ,x_2 ...,x_n| r_1,...r_{n-1},r_{n}> 
 \end{equation}
 Indeed, one of the $r_i$ is redundant. We keep it here so as to better compare it with the formulus appearing in the matrix theorem.
 Let   $\alpha: G(K)\longrightarrow G(K)/[G(K),G(K)]\cong \mathbb{Z}   $ be the abelianization of the fundamental group. Define the Alexander matrix of $G(K)$ to be
 \begin{equation}
 	A=\bigg(\alpha_* \bigg( \frac{\partial r_i}{ \partial x_j}  \bigg)\bigg) \in M(n\times n ;\mathbb{Z}[t;t^{-1}]  . ) 	
 \end{equation}
 Then the Alexander polynomial is defined (up to $\pm t^s$) to be 
 
 \begin{equation}\label{al}
 	\frac{detA^{i,j}}{t-1}
 \end{equation}
 where $A^{i,j}$ is obtained from $A$ by deleting the i-th row and j-th column(here i,j can be arbitrary) .
 \par
 Let $G$ be an unoriented diagram, with vertices $v_1, v_2,...,v_n$. Each edge $e_i$ is labelled with a complex number $x_i$.
 Define an $n\times n$ matrix $A(G)$:
 \begin{equation}
 	A(G)_{i,j}= \begin{cases}
 		\text{sum of weights on edges having $v_i$ as a vertex} & \text{if }  i\neq j \\
 		-\{ \text{sum of weights on edges from $v_i$ to $v_j$}  \}  & \text{ if }i=j \\
 	\end{cases}
 \end{equation}
 
 \par 
 
 Let $T \subset G$ be a tree with edges $e_{i_1},...e_{i_k}$, define the weight of $T$ to be
 \begin{equation}
 	\vert T \vert = x_{i_1}x_{i_2}...x_{i_k}
 \end{equation}
 and define the tree polynomial to be 
 \begin{equation}
 	\bigtriangledown_G =\sum \vert T \vert
 \end{equation}
 where the sum is over all maximal trees.
 The (unoriented) matrix tree theorem says that
 \begin{equation}\label{ma}
 	\bigtriangledown_G = Det( \hat{A}(G))
 \end{equation}
 where $\hat{A}(G)$ is obtained from $A(G)$ by deleting the i-th row and i-th column.(i can be arbitrary)

 \par
 Noticing the similarity between (\ref{al}) and (\ref{ma}), we tried to give a new interpretation for the Alexander polynomial using the matrix tree theorem.
 Indeed, we need a oriented version, and the concept of a tree has to be replaced by an arborescence.

 \par 
 \subsection{Matrix tree theorem for oriented weighted graphs}
 Let $G=(V,E)$ be a directed graph, i.e. each edge is assigned with an orientation. For $i,j \in V$, we say $j$ is an outneighbor of $i$, if there is an edge $i \rightarrow j$. An edge $e$ is an outedge of $v \in V$ if it starts at $v$. Furthermore there is an assignment of weights $w: E \rightarrow \mathbb{C}$ to the edges.

 Define the \textbf{outdegree} of a vertice $v$ with respect to the weight $w$ to be the sum of weights on all the outedges of $v$, and denoted by $deg^+_w(v)$. In particular, if the weight is trivial, then $deg^+_w(v)$ is just the number of edges starting at $v$.
 
 Similarly, we can define the indegree $deg^-_w(v)$ of a vertice $v$. 
 
 The \textbf{Laplacian} of G with respect to the weight $w$ is an matrix $L=(l_{ij})_{n \times n}$, where $n=|V|$:
 \begin{equation}
 	l_{ij}=\begin{cases}
 		deg^+(i) & \text{ if } i=j \\
 		-\sum _e w(e)  \text{ where the sum is over edges $i \longrightarrow j$} & \text{ if }i \neq
 		j  \\    \end{cases}
 \end{equation}
 
 An \textbf{arborescence} of $G$ with roots $v_1,v_2,...,v_k \in V$ is a tree such that every vertex other than the roots has out-degree one, and the roots have out-degree 0, or equivalently, every vertex has one and only one (oriented) path to a root. Denote the set of all arborescences with root $v_1,... v_k$ by $A_G(v_1,...v_k)$.

 \begin{theorem} (Directed Miltigraph matrix tree theorem)
 	Let $G=(V,E)$ be an oriented multigraph, with edge weight $w:E \longrightarrow \mathbb{C}$, and $L$ is the Laplacian. Denote by $L_{i_1,...,i_k}$ the matrix L removing the $i_1,...,i_k-th$ rows and columes.
 	Then 
 	\begin{equation}\label{mtt}
 		Det(L_{i_1,...,i_k})= \sum_{A \in A_G (v_{i_1},...,v_{i_k})} \omega (A)
 	\end{equation}
 \end{theorem}
 We give a proof here since some of the arguments are useful in the next section. See \cite{MR480115}.
 \begin{proof}
 	We prove the theorem by induction on $n-k$. The statement is true if $n=k$.
 	
 	Give a weight $x_i$ for each vertex $v_i$, i=1,2,...,m, $m=|V|$. Define a new weight 
 	\begin{equation}
 		\begin{aligned}\omega_{x_1,...x_m}:&E \longrightarrow \mathbb{C} \\ &e \longrightarrow \omega(e) \cdot\text{(weight on the end of e)} \\
 		\end{aligned}
 	\end{equation}
 	
 	Define the Laplacian and the weight of a arborescence under this new weight, and if there is no ambiguity, still denote them using the original notations. Then the original one becomes the special value for $x_1=...=x_m=1$. 
 	\\
 	
 	Now both sides of (\ref{mtt}) are degree $n-1$ polynomials. We want to show that every monomial appearing in the two sides is of degree 0 in some $x_j, j \neq i_1,...i_k$(j dependent on the monomial). Indeed, the degree of $x_j$ in $\omega(A)$ is just the trivial indegree $deg^-(v_j)$ of the vertice $v_j$.

 For $A \in A_G(v_{i_1},...v_{i_k})$, there is a vertex $v$ of $A$, $v \neq v_{i_1},...,v_{i_k}$, and $deg^-(v)=0$.
 
 Since $det(L_{i_1,...i_k})|_{x_{i_1}=...=x_{i_k}=0}$ divides $det(L)|_{x_{i_1}=...=x_{i_k}=0}$, $det(L)=0$, we have $det(L_{i_1,...i_k})|_{x_{i_1}=...=x_{i_k}=0}=0$. This means $det(L_{i_1,...i_k})$ has no terms dependent on all $x_i, i \neq i_1,...i_k$.
 
 Hence we only have to prove the equality "locally", for the parts independent of $x_i$, for each $i \neq i_1,...,i_k$.
 \\
 
 The part independent of $x_i$ on the left is
 \begin{equation}
 	Det(L_{i_1,...i_k,i}) \cdot(\sum_{j \neq i}L_{ji} x_i)
 \end{equation}
 
 Let $G'$ be the subgraph of $G$ generated by the vertices other than $v_i$, with Laplacian $L'$. Then $L_{i_1,...i_k,i}=L'_{i_1,...i_k}$.
 $\omega_A$ is independent of $x_j$if and only if the only edge connecting $x_j$with the remaining parts of $A$
 is an edge starting from $A$. Hence the right part independent of $x_j$ is
 \begin{equation}
 	\sum_{A \in A_G(i_1,...i_k)} \omega_A |_{x_j=0}= \sum_{A' \in A_{G'}(i_1,...,i_k)}\omega_{A'} \cdot  \sum_{j \neq i} L_{ji}x_i 
 \end{equation} 
 
 By induction 
 \begin{equation}
 	\sum_{A \in A_G(i_1,...i_k)} \omega_{A'}= Det(L'_{i_1,...i_k,})= Det(L_{i_1,...i_k,i}).
 \end{equation}
 
 The proof is completed.
 
\end{proof}
 
 \subsection{Application to the Alexander polynomial}
 
 Let $D$ be a diagram of a knot $K$, with arcs $A_1, A_2,...A_n$. Cut off some of them, say $A_{i_1},...A_{i_k}$, then $D$ becomes a k-string tangle $D_{i_1,...i_k}$. The arborescences in $D$rooted at $A_{i_1},...A_{i_k}$ are exactly  the trees in $D_{i_1,...i_k}$.
 And let $M_{D(i_1,...i_k)}$ be the matrix $M$ of $D$ negeleting the $i_1,...i_k-$ rows and columes. Using the matrix tree theorem for oriented graphs, we have:

  \begin{theorem}
 	\begin{equation}
 		M(D(i_1,...i_k))= \sum_{A \in T(D_{i_1,...i_k})} \omega_A
 	\end{equation}
 	where the sum is over all trees in $D_{i_1,...i_k}$.
 \end{theorem}
 
 In particular, if $k=1$, then
 \begin{corollary}
The Alexander polynomial of $K$
\begin{equation}
\bigtriangleup_K(t) = \sum_{A \in T(D_i)} \omega_A ,
\end{equation}
here $i$ can be any arc.

 \end{corollary}

Furthermore, taking $t=-1$, we get a formulus for the determinant of a knot:
 \begin{corollary}
 	\begin{equation}
 		det(K)=\bigtriangleup_K(-1)=\sum_{A \in T(D_i)} (-1)^{\alpha(A)}\cdot2^{\beta(A)}
 	\end{equation}
 where $\alpha(A),\beta(A)$ are respectively the number of go-straights and jump-ups of $A$.
 \end{corollary}

 \section{The Ihara zeta function and the determinant formula}

 \begin{figure}
 	
 	\centering
 	\includegraphics[scale=0.1]{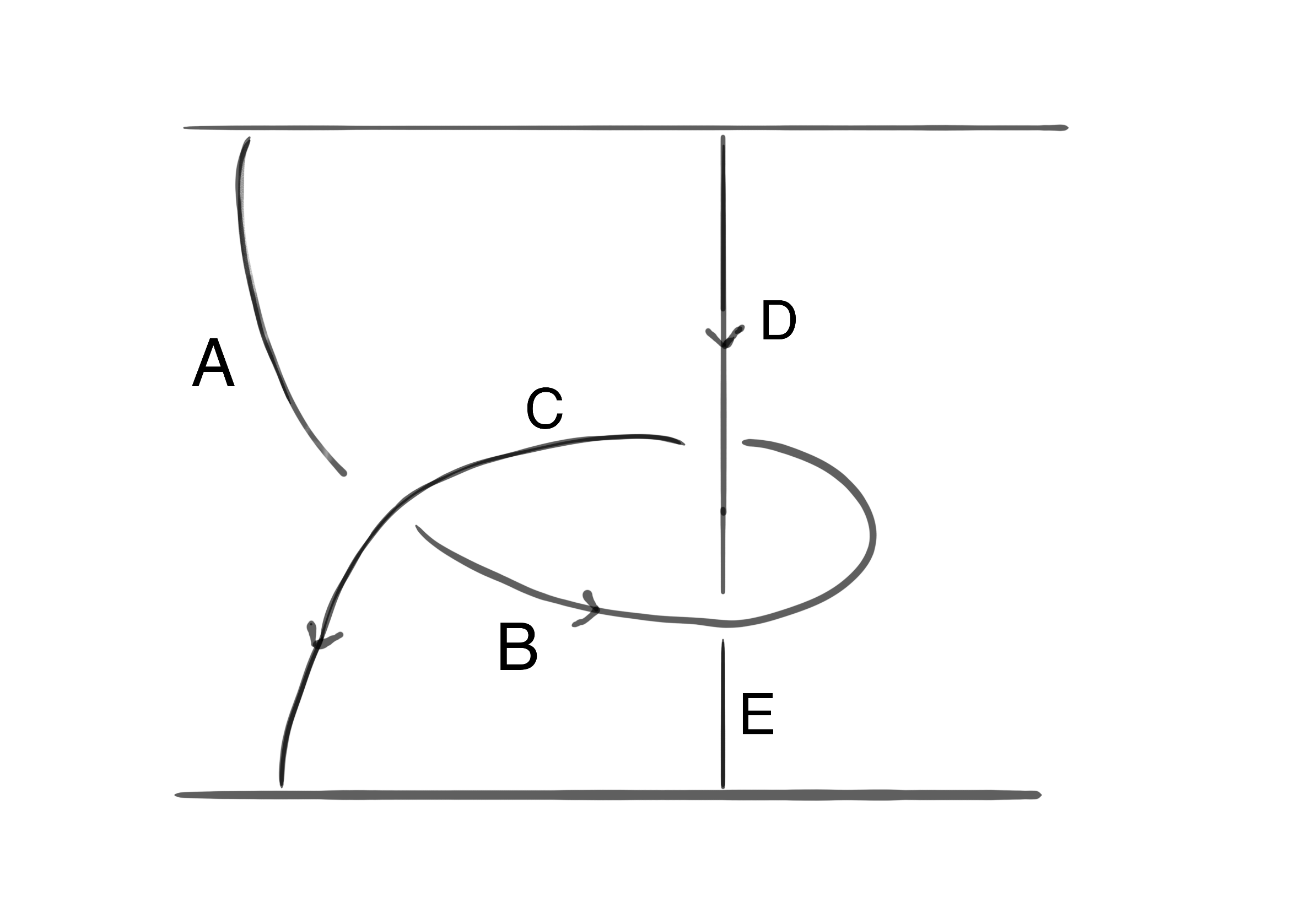}
 	\caption{an oriented tangle}
 	\label{fig1}
 \end{figure}

 \subsection{The arc diagram of a tangle} \label{ad}
 A tangle is a compact 1-manifold (with boundary) properly embedded in $\mathbb{R} \times \mathbb{R} \times [0,1] $ such that the boundary of the embedded 1-manifold is a set of distinct points in $ \{0\} \times \mathbb{R} \times \{0,1\} $. Two tangles are called isotopic if they are related by an isotpy of $\mathbb{R}^2 \times [0,1]$ fixing its boundary points.
 \\

 	Let $T$ be an oriented tangle, we assign an oriented weighted diagram $\textbf{Diag(T)}$ to $T$ as follows.
 	\par
 	
 	The vertices of $Diag(T)$ are in 1-1 correspondence with the arcs between any two adjoint undercrossings or boundary points. We denote the vertice corresponding to  the arc $C$ by $f(C)$. Let $C_1,C_2$ be two arcs on $T$. There is an edge from $f(C_1)$ to $f(C_2)$ with weight $t_1 \ (t_2)$ if $f(C_1)$ goes under a positive-oriented (negative-oriented) crossing point to $f(C_2)$, and there is an edge from $f(C_1)$ to $f(C_2)$ with weight $s_1 \ (s_2)$ if $f(C_1)$ jumps up at a positive-oriented (negative-oriented) crossing point to $f(C_2)$.

 For an oriented tangle $T$ with n arcs $A_1,...,A_n$, define an $n \times n$ matrix $M_T$ with
 \begin{equation}
 	(M_T)_{i,j}=\begin{cases}
 		
 		\text{the weight on the edge} & \text{if there is an oriented edge from $A_i$ to $A_j$}, \\
 		0 & \text{otherwise} \\
 		
 	\end{cases}
 \end{equation}
 And we call $det M_T$ the determinant of $T$.
 Let $D$ be a diagram for the knot $K$, with arcs $A_1,...,A_n$. Cut some arcs $A_{i_1},...A_{i_k}$ so we get a new tangle, denoted $D(i_1,...,i_k)$. As has been pointed out in Section 2, the Alexander polynomial is just 
 \begin{equation}
 	\frac{Det M_{D(i)}}{t-1} 
 \end{equation}
 for any i=1,2,...,n.
 
 \begin{figure}
 	
 	\centering
 	\includegraphics[scale=0.1]{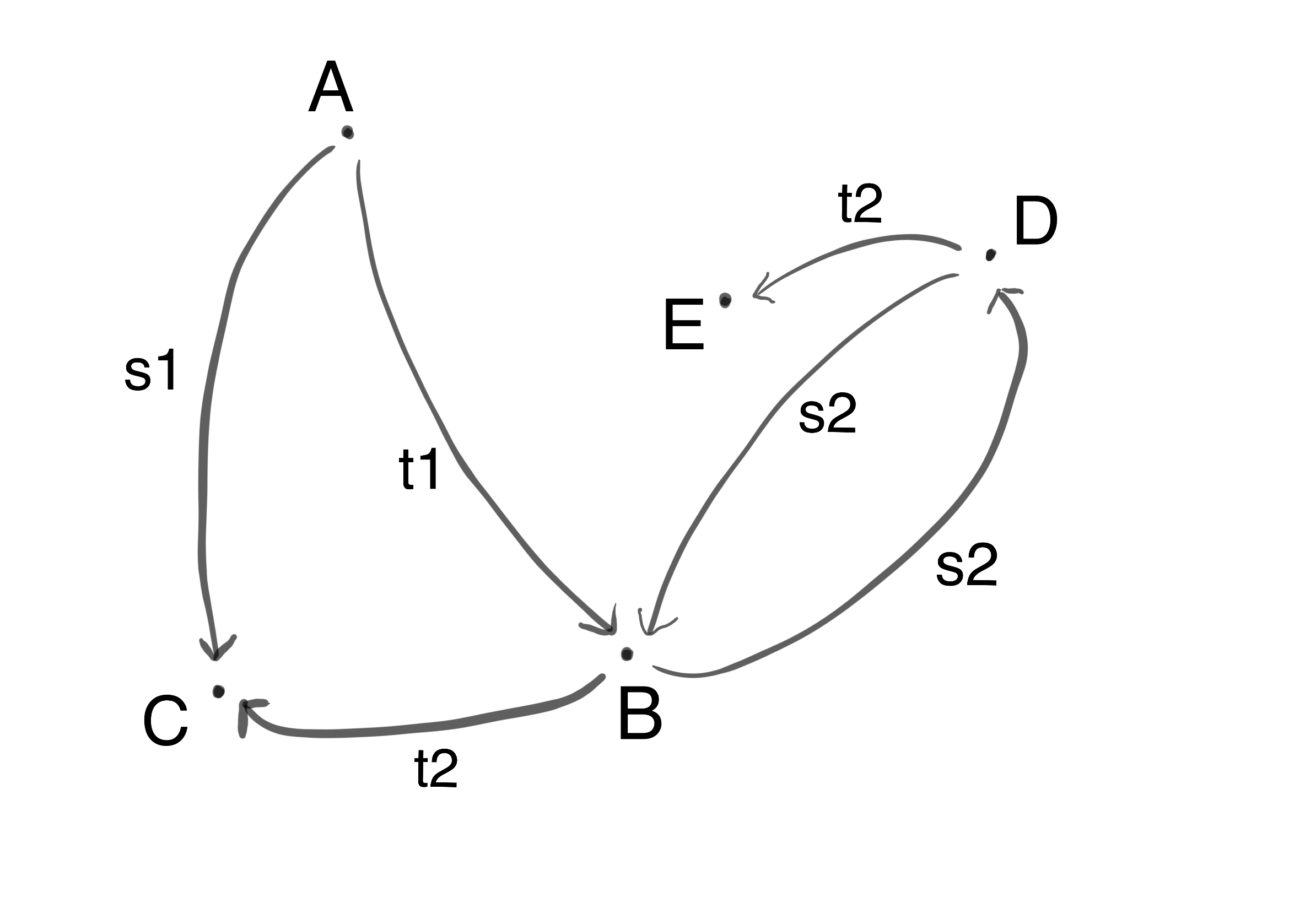}
 	\caption{The arc diagram of the oriented tangle in Figure \ref{fig1}}
 	\label{fig2}
 \end{figure}
	
	\subsection{The Ihara zeta function}
	First we review some basic notions on graphs.The terms are adopted from \cite{terras_2010}
	Let $X$ be an oriented graph with edge set $E$, vertex set $V$.
	A path $C=a_1...a_s$, where $a_j$ is an oriented edge of $X$, is said to have a \textbf{backtrack} if $a_{j+1}=a_j^{-1}$ for some $j=1,...,s-1$. We say $C$ has a \textbf{tail} $a_s=a_1^{-1}$; It is \textbf{closed} if the starting vertex is the same as the terminal vertex. The closed path $C$ is called a \textbf{prime} path if it has no backtrack or tail and $C \neq D^f$ for $f >1$. For the closed path $C=a_1...a_s$, the equivalence class $[C]$ means the following
	\begin{equation}
[C]={a_1...a_s, a_2...a_sa_1,...a_sa_1...a_{s-1}}
	\end{equation}
A \textbf{prime} in the graph $X$ is an equivalence class $[C]$ of prime paths.
Let $D$ be an arc diagram constructed in  Section \ref{ad}. By definition, for any $v_1,v_2$, there is at most one edge from $v_1$ to $v_2$. Define the adjacency matrix for $D$ to be an $n \times n$ matrix (n= the number of verticecs in $D$, and the vertices are denoted by $v_1,...v_n$).

\begin{equation}
	D_{i,j}= \begin{cases}
		\text{the weight on the edge} &  \text{if there is an edge from $v_i$ to $v_j$} \\
		0 & \text{otherwise}
	\end{cases}
\end{equation}

Given a closed path $C$ in $X$, which is written as a product of oriented edges $C=a_1...a_s$, the edge norm of $C$ is $N_E(C)=\omega_{a_1a_2}\omega_{a_2a_3}...\omega_{a_{s-1}a_s}\omega_{a_sa_1}$, where $\omega_{ab}$ denotes the weight of the edge from $a$ to $b$.
The \textbf{Ihara zeta function} of $X$ with weight $W$ is
\begin{equation}\label{zeta}
\zeta(W,X) = \prod_{[P]}(1-W(P))^{-1}=1+\sum W(C_1)W(C_2)...
\end{equation}
where the product is over all primes in X, and the sum is over all tuples of closed paths without backtracks or tails.Here we assume that all $\omega_{ab}$ are sufficiently small for convergence.

For convenience, we call the Ihara zeta function of $diag(T)$ to be the zeta function of the tangle diagram $T$, and denoted by $\zeta_T(t_1,t_2;s_1,s_2)$.
	
	\begin{lemma}\label{lemma}
(1)Let $T_1,T_2$ be two 1-string oriented tangle diagrams. Suppose they have the same orientation so that they can be composed. Then 
\begin{equation}
	\zeta_{T_1T_2}= \zeta_{T_1} \cdot \zeta_{T_2}
\end{equation}
here $T_1T_2$ is any composition of $T_1$ and $T_2$

(2) Let $T$ be an oriented tangle diagram, and $T^{{(n)}}$ be its n-cable. If we assign $s_1= 1-t_1,$, then
\begin{equation}
\zeta_{T^{(n)}}(t_1^{\frac{1}{n}},t_2^{\frac{1}{n}};1-t_1^{\frac{1}{n}},1-t_2^{\frac{1}{n}})=\zeta_T(t_1,t_2;1-t_1,1-t_2)
\end{equation}
	\end{lemma}

\begin{figure}
	
	\centering
	\includegraphics[scale=0.1]{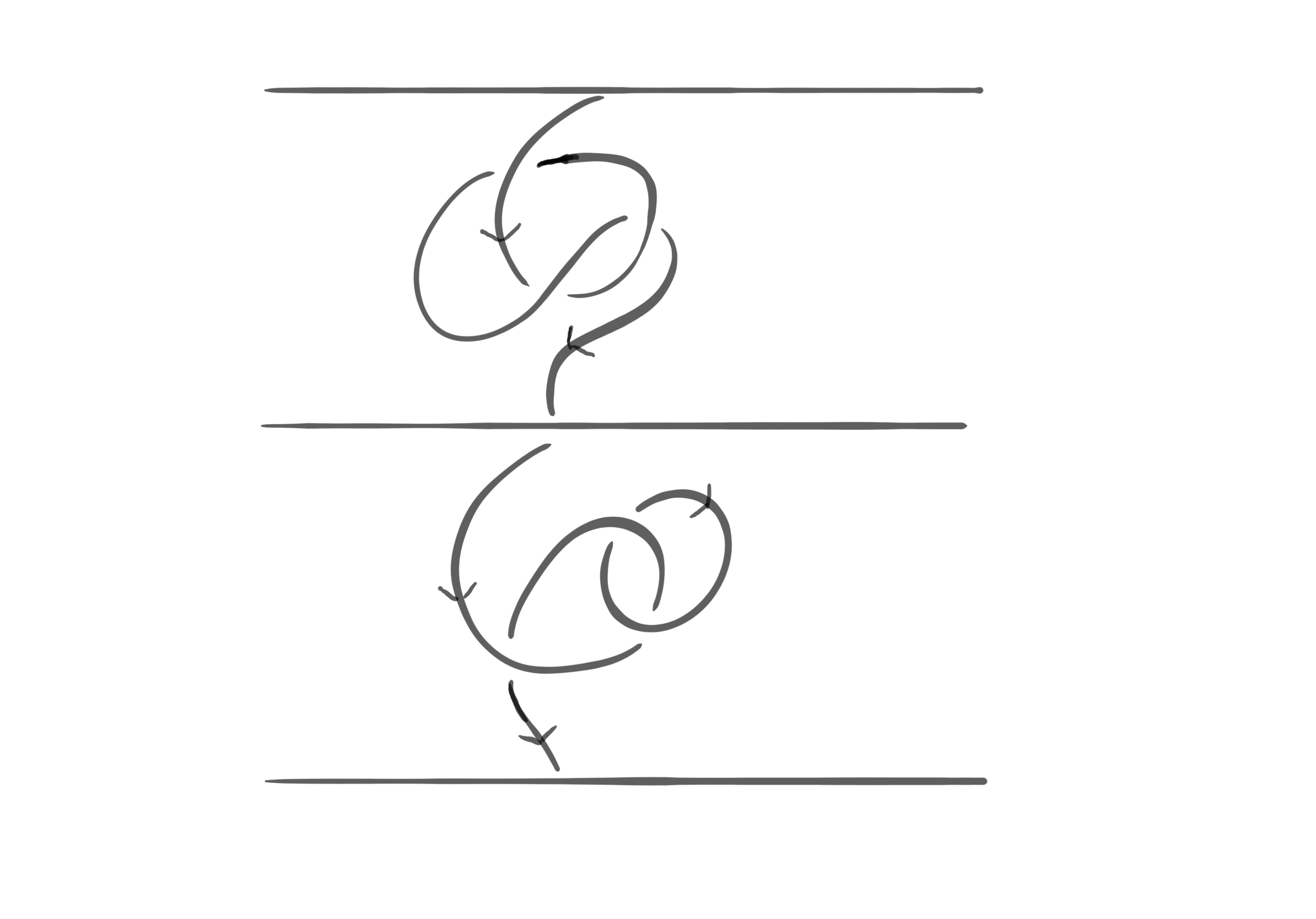}
	\caption{ The composite of two 1-string tangle}
	\label{fig4}
\end{figure}

\begin{proof}
(1)It is easily seen that the set of primes in a composition $T_1T_2$ consists of primes in $T_1$ and primes in $T_2$.

(2)Let $C$ be a cycle on $T$. There are $n^c$ cycles corresponding  to $C$ on $T^{(n)}$, where $c$= the crossing number of T. These $n^c$ cycles satisfy the property that when $c$ comes across under a crossing(jump up at a crossing), they comes  across under(jump up) at the corresonding crossings too.

When $C$ comes across a positive undercrossing, it contributes weight $t_1$ to the whole cycle; meanwhile a corresponding cycle $C'$ on $T{(n)}$ have to come across $n$ undercrossings, comtributing to $(t_1^\frac{1}{n})^n=t_1$ ,the same as above;
When $C$ jumps up at a positive crossing, it contributes weight $1-t_1$ to the whole cycle, and a corresponding cycle $C'$ on $T^{(n)}$ has $n$ choices: It can first come across $k$ under-crossings and then jump up, $k=0,1,...,n-1$. Hence the sum of weights here is 
\begin{equation}
	(1-t^{\frac{1}{n}})(1+t^{\frac{1}{n}})(1+t^{\frac{2}{n}})+...+t^{\frac{k}{n}} =1-t^n
\end{equation}

also equals to the above.
Hence we have proved 
\begin{equation}
\sum W(\widetilde{c}) =W(c)
\end{equation}
where the sum is over all cycles corresponding to $c$ in $T^{(n)}$.
By the second equality in(\ref{zeta}), the above argument shows that our statement is true.

\end{proof}

\begin{figure}
	
	\centering
	\includegraphics[scale=0.1]{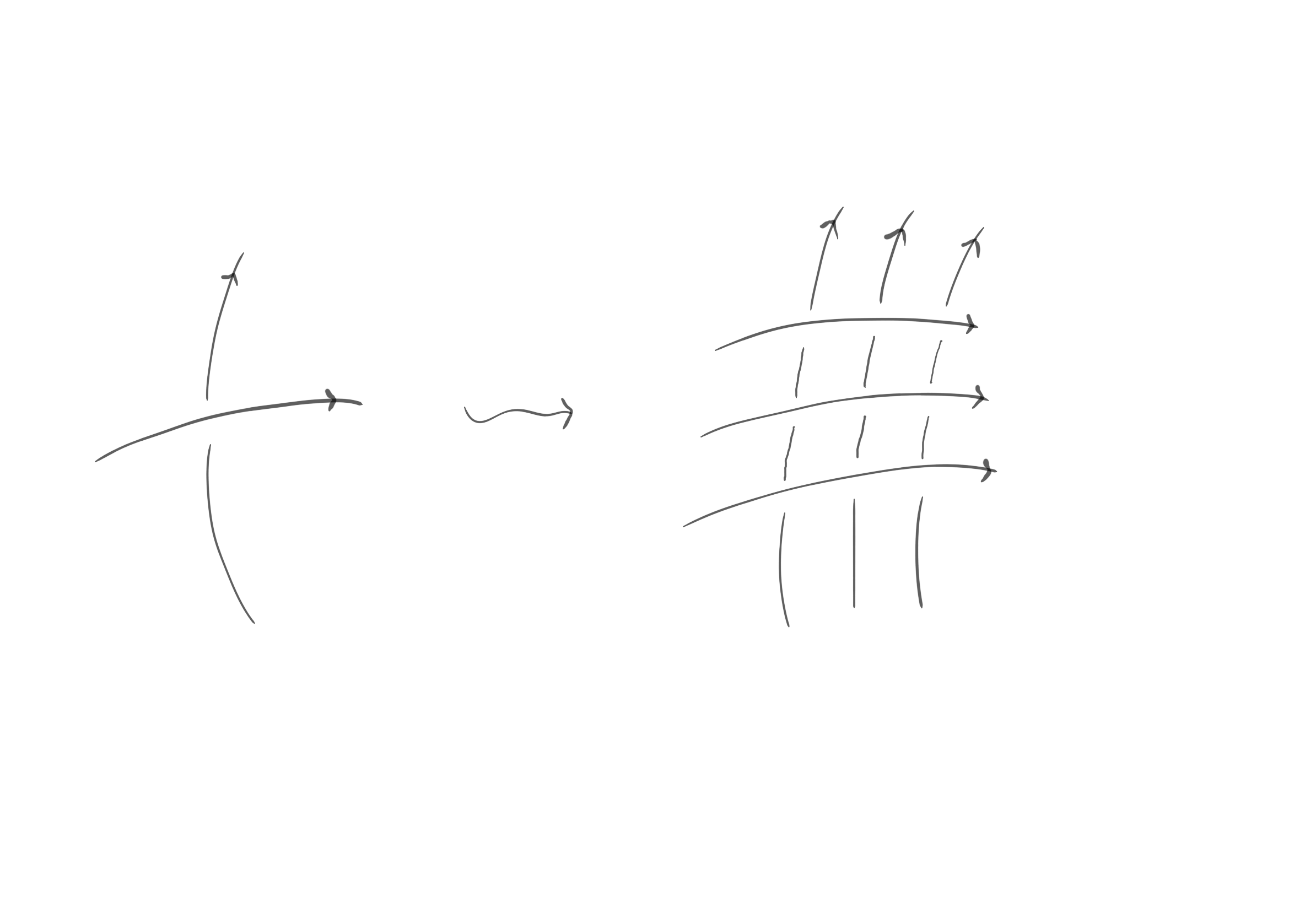}
	\caption{the local behaviour near a crossing point under a cabling}
	\label{fig4}
\end{figure}

\subsection{The determinant formula}
	
	\begin{theorem}
Let $X$ be a weighted oriented graph with weight matrix $W$. Then
\begin{equation}
\zeta(W,X) = det(I-W)^{-1}
\end{equation}
	\end{theorem}
This is a special case for of general Foata-Zeilberger formula, see also \cite{foata}.

We will focus on the following case: Recall that we define a weighted oriented diagram for a knot diagram $K$, with weights belonging to $\{t_1,t_2,s_1,s_2\}$. Now let $t_1=t,t_2=t^{-1},s_1=1-t^{-1},s_2=1-t^{-1}$, then the $n \times n$ edge matrix $W'$ of $Diag(K)$ becomes the matrix appearing in the Fox-calculus-definition of the Alexander polynomial.

Let W be an $(n-1) \times (n-1)$ matrix obtained from $W$ by deleting some row and some column. Then $det(I-W)$ is the Alexander polynomial of the knot. By the determinant formula, the Alexander polynomial has the form
\begin{equation}
\bigtriangleup_K(t) = \prod_{[P]}(1-N_E(P))^{-1}=\sum_{C}(1-N_E(C))
\end{equation}
where the product is over all primes of the diagram $Diag(K)$, and the sum is over all cycles without backtracks or tails.

\textbf{Remark.} Let \{$A_1,A_2,...A_n$\} be the set of arcs on the knot diagram $K$. Now we cut some arc, say $A_1$, so that $K$ becomes a 1-string tangle $K_1$,  and $A_1$ is broken into two arcs, $A_1^{'}$, $A_1^{''}$, 
 one initial and one terminal. Start walking along $K_1$. When we come across a positive crossing point, we have probability $t^\varepsilon$ to continue moving( along the orientation) , or jump up with probability $1-t^\varepsilon$, where $\varepsilon =\pm 1$ is determined by whether the crossing point is positive or not. The weight of a route on $K_1$ is by definition the multiplication of the possibilities at each crossing points. Then the above formula shows that the Alexander polynomial of $K$ equals the sum of weights of all the closed paths on $K_1$.

	\section{Some corollaries}
	
	We give some properties of the Alexander polynomial from this viewpoint, although they can be deduced from other methods, and are standard context in textbooks.
	
	\begin{lemma}\label{l1}
Let $T$ be a 1-string oriented tangle, with two ends denoted by $A^{'}, A^{''}$. Denote by $P(A^{'}, A^{''})$ the set of all paths from $A^{'}$ to $A^{''}$(Recall that when we say a path on the tangle, we mean a path on the corresonding arc diagram of the tangle). The weight $\omega(P)$ of a path $P$ is defined as in Remark $1$.
Then
\begin{equation}
\sum_{P(A^{'},A^{''})}\omega(P)=1
\end{equation}
	\end{lemma}
	\begin{proof}
		See \cite{MR1609599} \&3. Corollary 1.
	\end{proof}

	\begin{corollary}
		(1) Let $K_1, K_2$ be two knots, then
		\begin{equation}
			\bigtriangleup_{K_1+K_2}(t)=\bigtriangleup_{K_1}(t)\bigtriangleup_{K_2}(t)
		\end{equation}
	
	(2) Let $K_1, K_2$ be two knots, and K is a link as follows:

	then\begin{equation}
		\bigtriangleup_{K}(t) =(1-t^{-1})\bigtriangleup_{K_1+K_2}(t)
	\end{equation}
Of course this result can be obtained directly from the skein relation of the Alexander polynomial.

	(3)If $L$ is a split link, then $\bigtriangleup_L (t)=0$.
	\end{corollary}
	
	(4) Let $K$ be a knot and $K'$= the satellite knot of $K$ with pattern $P$.
	Then\begin{equation}
		\bigtriangleup_{K^{'}}(t)=\bigtriangleup_P(t) \cdot \bigtriangleup_{K}(t^n)
	\end{equation}
	
	\begin{proof}
		(1) Let $T_1, T_2$ be the 1-string tangle obtained respectively from $K_1, K_2$ by cutting some arc, and let $T$ be a composite tangle of $T_1, T_2$. Then $K$ is the closure of $T$. We have proved in Lemma \ref{lemma} that $\zeta_{T}=\zeta_{T_1}\zeta_{T_2}$, hence $\bigtriangleup_{K_1+K_2}=\bigtriangleup_{K_1}\bigtriangleup_{K_2}$.
		\\
		
		(2) Cut an arc and obtain a tangle $T$ as follows($T_1,T_2$ are still the 1-string tangle s from $K_1,K_2$, as above):
		
		The primes on $T$ consists of: (a)primes on $T_1$ ; (b)primes on $T_2$; (c) primes which contain the remaining arc of $K_2$ from $T_2$.
		By Lemma \ref{l1}, the total contribution of paths starting and ending at a same point is 1. Notice that each path corresonding to (c) comes across a negative undercrossing additionally, hence their contribution to the zeta function is $1-t^{-1}$. Since the factors in the zeta function of (a) and (b) are respectively the zeta function of $K_1$ and $K_2$, the proof is completed.
		\\
		
		(3) Suppose $L= L_1 \cup L_2$, and $L_1$ can be separated from $L_2$. Cut some arc on $L_1$ so that it becomes a tangle $T_1$   Then
		\begin{equation}
			\bigtriangleup_L(t)= \prod_{c \in T_1}(1-W(c_1)) \cdot \prod_{c_2 \in L_2}(1-W(c_2)) 
		\end{equation}
		However, $\prod_{c_2 \in L_2}(1-W(c_2))$= the determinant of the matrix of $L_1$, which has one row a linear combination of the others, hence is equal to 0.
		\\
		
		(4) We draw the diagram of $K^{'}$ as above, where the "nontrivial" part of $P$ and $K$  are separated. To determine the relationship between the Alexander polynomials of $K, P$ and $K^{'}$, we have to identify the primes of $K^{'}$ from that of $K, P$. 
		
		First we discuss the primes of the right hand part on the diagram. It is an n-cable of the 1-string tangle obtained from $K$. By Lemma \ref{lemma}, the factors in the zeta function $\zeta_{K^{'}}(t)$ corresonding to the primes on this part is just $\zeta_K(t^n)$. When we reglue the two parts together, except for the cycles carried by $P$ and $K^(n)$, new cycles appear. They are the cycles containing some arc linking the two "nontrivial" parts on the diagram. However, these cycles contribute nothing to our zeta function of $K'$, as illustrated in Lemma $\ref{l1}$. Hence the zeta function of $K^{'}$ is just the product of $\bigtriangleup_P(t)$ and $\bigtriangleup_K(t^n)$.
		
	\end{proof}

	\begin{figure}
		
		\centering
		\includegraphics[scale=0.1]{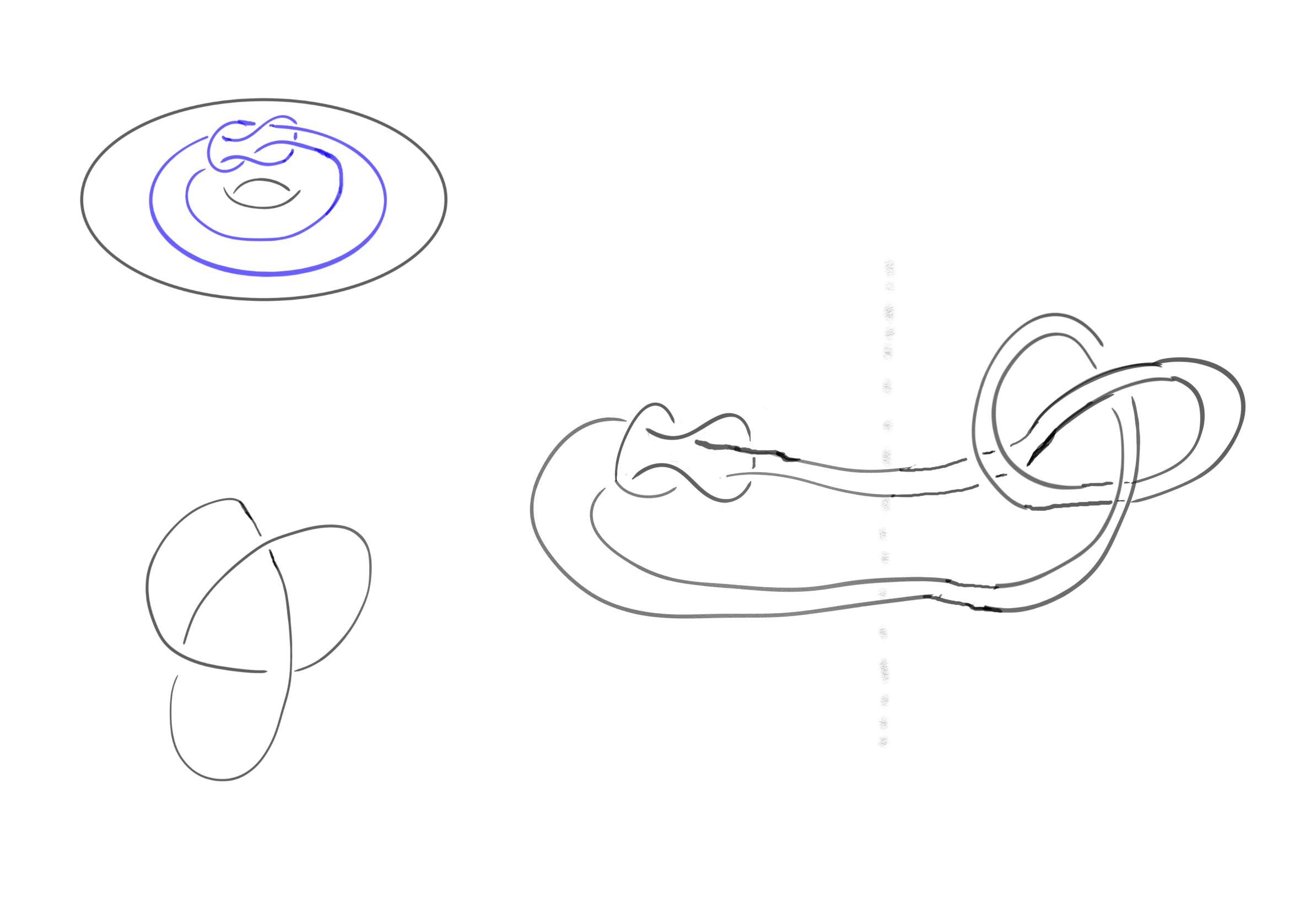}
		\caption{A satellite knot of the trefoil knot}
		\label{fig5}
	\end{figure}

	\section{Twisted Alexander invariants}
	In this section we discuss possible generalization of zeta function formula to twisted Alexander invariants. First we consider the twisted Alexander polynomials, see \cite{Kitano}.
	
	We recall the Fox-calculus definition of twisted Alexander polynomials.
	
	Let $\Gamma = \pi_1(S^3 -K)$ be the knot group of a knot $K$. It has the Weigtinger presentation
	\begin{equation}
\Gamma =<x_1,...x_n | r_1,..., r_{n-1}>
	\end{equation}
	where $r_i$ is of the form $x_j^{\epsilon}x_kx_j^{- \epsilon} x_i$.
	
	Let $F_k =<x_1,...,x_k>$ be the free group genereted by $k$ elements, and $\phi : F_k \longrightarrow \Gamma $ the canonical surjective homomorphism. It induces a homomorphism between the group rings:
	\begin{equation}
q : \mathbb{Z}[F_k] \longrightarrow \mathbb{Z}[\Gamma]
	\end{equation}
	Let $\rho : G(K) \longrightarrow GL(m;\mathbb{F})$ be a representation. Write 
	\begin{equation}
\rho_* : \mathbb{Z}G(K) \longrightarrow \mathbb{Z}GL(m;\mathbb{F}) \subset GL(m;\mathbb{F})
	\end{equation}
for the ring homomorphism induced by $\rho$.

Let $\alpha : G(K) \longrightarrow$ $<t>$ $ \cong  \mathbb{Z}$ be the abelization.
Denote by $\Phi $ the composite of $\mathbb{ZF}_n \longrightarrow \mathbb{Z}G(K)$ and $\rho_* \otimes \alpha_* : \mathbb{Z}G(K) \longrightarrow GL(m;\mathbb{F}) \otimes\mathbb{Z}  [t,t^{-1}] \subset GL(m; \mathbb{F}[t,t^{-1}]):$
\begin{equation} \label{Phi}
\Phi : \mathbb{Z}[F_k] \longrightarrow M(2; \mathbb{F}[t,t^{-1}])
\end{equation}

 Define an $(n-1) \times n$ matrix 
	\begin{equation}
		A_{\rho}= \left( \begin{array}{cccc}
			\Phi(\frac{\partial r_1}{\partial x_1}) &  \Phi(\frac{\partial r_1}{\partial x_2}) & \ldots 
			\Phi(\frac{\partial r_{1}}{\partial x_n})
			\\
			\Phi(\frac{\partial r_2}{\partial x_1}) & \Phi(\frac{\partial r_2}{\partial x_2}) & \ldots
			\Phi(\frac{\partial r_{2}}{\partial x_n})
			\\ 
			\vdots & \vdots & \ddots \\
			\Phi(\frac{\partial r_{n-1}}{\partial x_1}) &
			\Phi(\frac{\partial r_{n-1}}{\partial x_2}) & \ldots
			\Phi(\frac{\partial r_{n-1}}{\partial x_n})  \\
		\end{array}   \right)
\end{equation}
called the twisted Alexander matrix of $G(K)= <x_1,...,_n | r_1, ...r_{n-1}>$ associated to the representation $\rho$. Denote by $A_{\rho, k}$ the $(n-1)\times(n-1)$  matrix obtained from $A_{\rho}$ by removing the $k$-th cloumn.

The twisted Alexander polynomial of $K$ for $\rho$ is defined to be 
\begin{equation}
	\bigtriangleup_{K,\rho}(t)= \frac{\det A_{\rho, k}}{\det \Phi(x_k -1)}
\end{equation}
here $k$ is chosen so that $\det \Phi(x_k -1) \neq 0$. Note that when $\rho$ is the trivial representation
 \begin{equation}
G(K) \longrightarrow GL(1;\mathbb{R})
 \end{equation}
the corresponding twisted Alexander polynomial becomes the usual Alexander polynomial.
\\

Using the random walk model on knot diagram, we can generalize the zeta function formula for twisted ALexander polynomials. However, this time we have to be careful about the notation of the arcs. 

Cut out the knot at some arc. Denote the arc on one strand by $x_1$, walk along the orientation, and denote the arcs in order by $x_2, x_3...$ along the way. We denote a crossing point by $a_k$ if the arc $x_k$ comes under  this crossing point, and $r_i$ is the relationship of the $x_j{'}s$ around the point $a_i$ .
Then $r_i$ is of the form $x_ix_jx_{i+1}^{-1}x_j^{-1}$.
In particular, 
\begin{equation}
	\frac{\partial r_i}{\partial x_i}=1,  \quad
	\frac{\partial r_i}{\partial x_{i+1}}= -x_ix_jx_{i+1}^{-1},\quad
	\frac{\partial r_i}{\partial x_{i+1}} =x_i -x_ix_jx_{i+1}^{-1}x_j^{-1}
\end{equation}

 Define a weight on the arc diagram of $K$ as follows. If the edge $e$ starts at $x_i$ and ends at $x_j$, then we give it weight wwith respect to the representation $\rho$
 \begin{equation}
W_{\rho}(e)=\begin{cases} \Phi(x_ix_jx_{i+1}^{-1}) &\text{if} \quad j=i+1 \\
\Phi(x_ix_jx_{i+1}^{-1}x_j^{-1} -x_i) & \text{if}  \quad j \neq i+1 
\end{cases}
 \end{equation}
where $\Phi$ is defined as (\ref{Phi}).
Now the twisted Alexander matrix equals to $I-B$, where $B_{i,j}$= the evaluation of $W_{\rho}$ on the edge from $x_i$ to $x_j$. Hence by the general determinant formula, we have 
\begin{equation}
	det A_{\rho,n} = det (I-B) = \prod_{[P]} (w- W_{\rho}(P))
\end{equation}
where the product is over all primes on the arc diagram.
\\

Finally we discuss the $L^2$- torsion. Instead of the determinant of a matrix, the $L^2$- torsion is defined as the Fuglede- Kardison determinant of an operator, analogous to the construction above. First we review some basic notions necessary for the definition of $L^2$- Alexander invariant.

Let $\Gamma$ be a group. Define 
\begin{equation}
	l^2(\Gamma) = \text{the Hilbert space }\{ \ \sum_{\gamma \in \Gamma} a_{\gamma} \gamma | a_{\gamma} \in \mathbb{C} \quad with \sum |a_{\gamma}|^2 < + \infty \ \}
\end{equation}
The inner product on the Hilbert space is :
\begin{equation}
< \sum_{\gamma \in \Gamma} a_{\gamma} \gamma , \sum_{\gamma \in \Gamma} b_{\gamma} \gamma > = \sum_{\gamma \in \Gamma} a_{\gamma} \overline{b}_{\gamma}
\end{equation}

$\Gamma$ acts on $l^2(\Gamma)$ by left multiplication. Define
\begin{equation}
	l^2(\Gamma)^{[n]} =\underbrace{l^2(\Gamma) \oplus ...\oplus l^2(\Gamma)}_n
\end{equation}
 We call it a free $\mathcal{N}(\Gamma)$- Hilbert module of rank $n$.Let $e_i$ be the unit element in the $i$-th copy of $l^2(\Gamma)$ in $l^2(\Gamma)^{[n]}$.
 Let $V$ be a Hilbert $\mathcal{N}(\Gamma)$-submodule in $l^2(\Gamma)^{[n]}$ (i.e. $V$ is a Hilbert space embedded in $l^2(\Gamma)^{[n]}$, with the induced $\Gamma$- action on it), which has a Hilbert basis $\{u_j\}_{j \in J}$, and $f: V \longrightarrow V$ is a positive endomorphism of this this Hilbert $\mathcal{N}$-module.
 \begin{definition}
Define the von Neumann trace to be
\begin{equation}
tr_{\mathcal{N}(G)}(f) = \sum_{j \in J} < f(u_j, u_j) >
\end{equation}
 \end{definition}
The von Neumann dimension of $V$ is defined as
\begin{equation}
	dim_{\mathcal{N}}(G) (V) =tr_{\mathcal{N}(G)} ( id: V \longrightarrow V)
\end{equation}

\begin{definition}
Let $f: l^2(G)^m \longrightarrow l^2(G)^n$ be a homomorphism of Hilbert $\mathcal{N}(G)$- modules. The spectral density function of $f$ is 
\begin{equation}
F(f): \begin{aligned} 
	&[0, +\infty] &  &\longrightarrow&  &[0, +\infty] &\\ &\lambda& &\mapsto&  &sup\{  dim_{\mathcal{N}(G)}(L) |L \subset l^2(G)^m \text{a Hilbert }  \mathcal{N}(G)  \text{-submodule of }&  \\ &&&& &  l^2(G)^m \text{ such that } ||Ax|| \le \lambda \cdot ||x|| \text{ for all }  x \in L \} &
\end{aligned} 
\end{equation}

\end{definition}

\begin{definition}
Let $A$ be an $n \times n$ matrix over $\mathbb{R}[\Gamma]$. It defines a map of Hilbert $\mathcal{N}(\Gamma)$- modules $ l^2(\Gamma)^{[n]} \longrightarrow l^2(\Gamma)^{[n]} $ , and let $F_A$ be the spectral density function.
Define   

\begin{equation}
	det_{\mathcal{N}(G)}(f) = \begin{cases} exp \int_{0^+}^{\infty} ln(\lambda) dF & \text{ if }      \int_{0^+}^{\infty} ln(\lambda) dF > -\infty \\
	0  &\text{if }  \int_{0^+}^{\infty} ln(\lambda) dF =-\infty
	\end{cases}
\end{equation}

We say $f$ is of determinant class if $\int_{0^+}^{\infty} ln(\lambda) dF >-\infty$

\end{definition}

\begin{lemma}\label{lemma3}
(1) If f is invertible, then 
\begin{equation}
det(f) =exp (\frac{1}{2} \cdot tr (ln(f^*f)))
\end{equation}

(2) If f: U $\longrightarrow$ U is an injective positive operator, then
\begin{equation}
\lim_{\epsilon \rightarrow 0^+} det (f+ \epsilon \cdot id_U) =det(f)
\end{equation}

(3)\begin{equation}
  det(f)= det(f^*) = \sqrt{det(f^*f)} =\sqrt{det(ff^*)}
\end{equation}
\end{lemma}

For a proof, see \cite{https://doi.org/10.1007/978-3-662-04687-6}.

Now let $\Gamma$ be a knot group,with representation $<x_1, x_2...,x_n|r_1,...,r_{n-1}>$ and $\alpha: \Gamma \longrightarrow U(1)$ be the abelization. Let $F_k={x_1,...,x_n}$, and $\mathbb{Z}[F_k]$ the free group. Let $\phi:\mathbb{Z}[F_k] \longrightarrow \Gamma $ the natural map.  The right multiplication of $\Gamma$ on $l^2(\Gamma)$ induces a map
\begin{equation}
	\rho_\Gamma : \Gamma \longrightarrow GL(l^2(G))
\end{equation}
They induce maps on the group rings:
\begin{equation}
	\begin{split}
(\alpha)_* : \mathbb{Z}\Gamma \longrightarrow \mathbb{Z}<t> \\
(\rho_\Gamma)_* : \mathbb{Z}\Gamma \longrightarrow  \mathbb{Z}GL(l^2(G))
\end{split}
\end{equation}
Tensor the two maps , we obtain:
\begin{equation}
(\alpha)_* \otimes (\rho_\Gamma)_* : \mathbb{Z}\Gamma \longrightarrow  \mathbb{Z}<t> \otimes \mathbb{Z}GL(l^2(G)) \subset \mathcal{N}(\Gamma)
\end{equation}

Let $\Psi$ be the composite
\begin{equation}
\Psi = (\alpha_* \otimes {\rho_\Gamma}_*) \circ \widetilde{\phi} :\mathbb{Z}[F_k] \longrightarrow \mathcal{N}(\Gamma)
\end{equation}

Now we can define the $L^2$- Alexander torsion.

Define an operator $A_{\rho_{\Gamma} \otimes \alpha} : l^2( \Gamma )^{[k-1]} \longrightarrow l^2( \Gamma )^{[k]}  $ such that
\begin{equation}
	A_{\rho_{\Gamma} \otimes \alpha, (i,j)} = \Phi (\frac{\partial r_i}{ \partial x_j})
\end{equation}
It is called the $L^2$- Alexander matrix. Furthermore, let
\begin{equation}
	A_{\rho_{\Gamma} \otimes \alpha}^j : l^2(\Gamma)^{[k-1]} 
	\longrightarrow l^2(\Gamma)^{[k-1]}
\end{equation}
 be the morphism obtained from $	A_{\rho_{\Gamma} \otimes \alpha}$ by removing the $j$- th column from its matrix form.
 \begin{definition}The $L^2$- Alexander torsion of a knot is defined to be
 	\begin{equation}
\bigtriangleup_{K}^{(2)}(t) := Det_{\mathcal{N}(\Gamma)} (	A_{\rho_{\Gamma} \otimes \alpha}^1)
 	\end{equation}
 \end{definition}

We cannot get a zeta function formula like the twisted Alexander case since, unlike the determinant, there is no direct relationship between the $L^2$- torsion of a matrix and its entries.

If $f=(f_{ij}): l^2(\Gamma)^{[n]} \longrightarrow l^2(\Gamma)^{[n]}$ is an invertible operator, we can get a similar formulation. We have to first modify slightly the definition of the Ihara zeta function.

Let $X= (V, E)$ be an oriented graph such that for any vertices $v_1, v_2 \in V$, there is at most 1 edge from $v_1$ to $v_2$. Therefore we can identify an edge by its starting and termianal points. Let $\Gamma$ be a group.
For each edge $e_{ij}$ we assign an injective operator of determinant class $f_{ij}: l^2(\Gamma) \longrightarrow l^2(\Gamma) $ as the weight $\omega(e_{ij})$ of $e_{ij}$.

Given a closed path $C= e_1e_2...e_s$ on $X$, define the weight of $C$ to be $W(C):= W(e_1)W(e_2)...W(e_s)$.
In contrast to the usual case, the composite of operators is not commutative, hence we cannot identify a closed path with its conjugacy class. In particular, the concept of a prime is invalid now. For any path $P$, we denote by $l(P)$ the length of $P$.
\begin{definition}
	Let $X$ be an oriented graph as above, with a weight $W$ of operators of determinant class on each edge. The Ihara zeta function of $X$ with respect to $W$ is
	\begin{equation}
		\zeta(W,X) = \prod_P(1-\frac{1}{l(P)} det_{\mathcal{N}(\Gamma)} W(P))^{-1}
	\end{equation}
where the product is over all prime closed paths(with a marked starting point) of $X$.
 \end{definition}

Let $D$ be a knot diagram for an oriented knot $K$ with arcs $x_1,...,x_n$. As in the twisted case, they are ordered in such a way that $x_{i+1}$ follows $x_i$ along the orientation. And a crossing point is denoted $a_k$ if it is the terminal point of $x_k$. Let $r_i$ be the relationship of the $x_j$'s around $a_i$.
Then $\frac{\partial r_i}{\partial x_i} =1$, hence $\Psi( \frac{\partial r_i}{\partial x_i}) =id $.

\begin{lemma}
	Suppose $f=(f_{ij})_{n \times n} : l^2(\Gamma)^{[n]} \longrightarrow l^2(\Gamma)^{[n]}$ is an invertible self-adjoint operator, and each $f_ij : l^2(\Gamma) \longrightarrow l^2(\Gamma) $ is invertible, then
	\begin{equation}
		det(I-f)^{-1} = \zeta (W,X)
	\end{equation} 
\end{lemma}

\begin{proof}
By Lemma \ref{lemma3}, we have
\begin{equation}
det_{\mathcal{N}(\Gamma)}(Id - f)= \exp tr_{\mathcal{N}(\Gamma)} \ln(I - f)
\end{equation}
Hence
\begin{equation}
	\begin{aligned}
		\log det_{\mathcal{N}(\Gamma)} (I- f)^{-1}& = - tr_{\mathcal{N}(\Gamma)} ln (I-f) 
		\\
		&=-tr_{\mathcal{N}(\Gamma)}( \sum_{m=0}^\infty \frac{f^m}{m})
		\\
	&	=-\sum_{m=0}^\infty \frac{tr_{\mathcal{N}(\Gamma)}f^m}{m}
	\end{aligned}
\end{equation}

On the other hand,
\begin{equation}
-\log \zeta(W,X)= \sum_P \sum_{j \ge 1} \frac{1}{j\cdot l(P)} det_{\mathcal{N}(\Gamma)} W(P)
\end{equation}
where the first sum is over all prime cycles.
It follows that 
\begin{equation}
	-\log \zeta(W,X) =\sum_C \frac{W(C)}{l(C)}
\end{equation}
where the sum is over all (not necessarily prime) closed paths (without bacjtracking or tails).

Let $e_i$ be the unit element in the $i$-th component of $l^2(\Gamma)^{[n]}$. Then $e_1, ..., e_n$ forms a basis for $l^2(\Gamma)^{[n]}$. By definition, 
\begin{equation}
	tr_{\mathcal{N}(\Gamma)} f^m = \sum_{i=1}^n <f(e_i)^m, e_i>
	=\sum_{i=1}^n \sum_{j_1,...,j_{k-1}} f_{j_{k-2}j_{k-1}}...f_{j_2j_3} \circ 
	f_{j_1j_2}\circ f_{ij_1}(e_i)
\end{equation}
where the second sum is over all $j_1, j_2,..., j_{k-1} \in \{1,...,n\}$.Note that $j_{kk}=0$ for each $k$.

\end{proof}

We can view a choice of $j_1,...,j_{k-1}$ as a path starting and terminating at $i$, i.e. a closed path fixed at $i$. Then 
\begin{equation}
\sum_{m \ge 1} \frac{tr_{\mathcal{N}(\Gamma)} f^m}{m} = \sum_{m \ge 1} \frac{det_{\mathcal{N}(\Gamma)} W(C)}{l(C)} =-\log \zeta(W,X)
\end{equation}
It follows that
\begin{equation}
	\log \zeta(W,X)=\log det_{\mathcal{N}(\Gamma)} (I- f)^{-1}
\end{equation}
so
\begin{equation}
	\zeta(W,X)= det_{\mathcal{N}(\Gamma)} (I- f)^{-1}
\end{equation}

In general	$A_{\rho_{\Gamma} \otimes \alpha}^j$ is not self-adjoint, hence we cannot apply the above lemma directly to the $L^2$ Alexander torsion. The zeta function expression must include entries of $({A_{\rho_{\Gamma} \otimes \alpha}^j})^* $, not only $A_{\rho_{\Gamma} \otimes \alpha}^j$. It seems difficult to utilize such an expression of $L^2$ torsion to find relations to the colored Jones polynomials.

\bibliography{zz}
\nocite{*}	
\end{document}